\documentclass[reqno,11pt]{amsart}
\usepackage[mathscr]{euscript}

\usepackage{amsmath,stmaryrd,xspace,amssymb,bm,bbm}

\usepackage{pstricks}
\usepackage{graphicx}
\usepackage{gastex}

\usepackage[bookmarks=false,colorlinks,linkcolor=mylinkcolor,citecolor=mycitecolor]{hyperref}
	\definecolor{mycitecolor}{rgb}{.1,.5,.1}
	\definecolor{mylinkcolor}{rgb}{.1,.1,.6} 
	\definecolor{myblue}{rgb}{.3,.3,.9}
	
	\definecolor{christophe}{rgb}{1,.2,.2}
	\definecolor{aaron}{rgb}{.2,.2,1}

\usepackage[margin=1.50in]{geometry}

\numberwithin{equation}{section}

\theoremstyle{plain}
	\newtheorem{theo}{Theorem}[section]
	\newtheorem{prop}[theo]{Proposition}
	\newtheorem{coro}[theo]{Corollary}
	\newtheorem{prob}{Problem}
	\newtheorem{lemm}[theo]{Lemma}

\theoremstyle{definition}
	
	\newtheorem{defi}[theo]{Definition}
	\newtheorem{exam}[theo]{Example}

\theoremstyle{remark}
	 
	\newtheorem*{rema}{Remark}

\newcommand{\demph}[1]{\emph{#1}}  

\newcommand{\auxtrans}[3]{T_{#1,#3}} 
\newcommand{\revtrans}[3]{T^{{}^{\,\mathsf{v}}}_{#1,{#3}^\ast}} 

\def\bs{\boldsymbol}
\def\qand{\quad\hbox{and}\quad}

\def\into{{\hookrightarrow}}

\def\shuf{\mathop{\sqcup\!\sqcup}}
\def\infl{\mathop{\bs\uparrow}}
\def\Mag{{\mathsf M}}

\def\img{{\mathsf{img}\,}}
\def\rnk{{\mathsf{rnk}\,}}

\def\circc{\mathop{\scriptstyle\stackrel{\circ}{}}}

\renewcommand{\red}{{\mathop{\mathsf{red}}}}

\def\sbar#1{{{\scriptstyle\,}\overline{#1}{\scriptstyle\,}}}

\def\inv{{}^{-1}}
\renewcommand{\arg}{{\bs\cdot}} 
\newcommand{\supp}{\operatorname{\mathsf{supp}}}
\newcommand{\suppx}[1]{\supp(#1)}

\def\lparen{(\!(}
\def\rparen{)\!)}
\def\lskew{\makebox[1.0\width]{$\hbox{$<$}\!\!\!({\scriptscriptstyle\,}$}}
\def\rskew{\makebox[1.0\width]{${\scriptscriptstyle\,})\!\!\!\hbox{$>$}$}}
\def\llangle{\langle\!\langle}
\def\rrangle{\rangle\!\rangle}

\def\K{{k}} 

\def\N{{\mathbb N}}
\def\Z{{\mathbb Z}}

\newcommand{\poly}[1]{\K#1}
\newcommand{\ser}[1]{{\K}^{#1}}

\newcommand{\monser}[1]{\K\llangle #1 \rrangle}

\newcommand{\mn}[1]{{\K\lparen#1\rparen}}
\newcommand{\ff}[2][0ex]{{\K\kern #1\lskew#2\rskew}}
\def\fg{{\Gamma}}
\newcommand{\fgx}[1]{{\fg(#1)}}	

\def\w{\omega}
\def\uu{\upsilon}

\def\ixx{{X\cup{X}^{-1}}}
\def\bxx{{X\cup\overline{X}}}
\def\XX{{{(\bxx)}^{\ast}}}

\def\vv{\fg}
\def\ee{\mathcal{E}}
\def\cc{\mathcal{G}}

\def\pp#1{{\wp(#1)}}
\def\F{{\mathfrak{F}}}
\def\FF#1{{\left[\F,#1\right]}}

\def\suff#1{{\mathscr{S}(#1)}}

\def\oo#1{{{#1}_{\mathrm{o}}}}

\renewcommand{\prec}{\mathrel{<_{{}_\fg}}}
\renewcommand{\succ}{\mathrel{>_{{}_\fg}}}

\title[Rational series and the {C}onnes operator]{
Rational series in the free group and the {C}onnes operator
} 

\author{Aaron Lauve} \address[Aaron Lauve]{
	Department of Mathematics \& Statistics \\ 
	Loyola University Chicago \\ 
	1032 W Sheridan Road \\ 
	Chicago, IL 60660 \\ 
	USA} 
	\email{lauve@math.luc}  \urladdr{http://www.math.luc.edu/$\scriptstyle\sim$lauve}

\author{Christophe Reutenauer}\address[Christophe Reutenauer]{
	LaCIM\\ 
	Universit\'e\! du\! Qu\'ebec\! \`a\! Montr\'eal\\ 
	Case Postale 8888, succursale Centre-ville\\ 
	Montr\'eal\! (Qu\'ebec)\! H3C 3P8\\ 
	CANADA} 
	\email{reutenauer.christophe@uqam.ca}  \urladdr{http://www.lacim.uqam.ca/~christo/}

\date{July 4, 2012}

\begin{document}

\begin{abstract}
We characterize rational series over the free group by using an operator introduced by  A. Connes. We prove that rational Malcev--Neumann series posses rational expressions without simplifications. Finally, we develop an effective algorithm for solving the word problem in the free skew field.
\end{abstract}

\maketitle


\section{Introduction}\label{sec: intro}

Fix a field $\K$ and a finite set of indeterminants $X$. The free commutative field $\K(X)$ is defined as the initial object in the category of (commutative) fields $F$ with embeddings $\K[X]\into F$. It also has a well understood realization as the field of rational functions:
\[
\K(X) = \big\{ {f}/{g} \mid f,g\in \K[X],\, g\neq 0 \big\}.
\]
The focus of this paper is the noncommutative counterpart to $\K(X),$ \emph{the free (skew) field} $\ff{X}.$ It has an analogous categorical definition involving embeddings of $\K\langle X\rangle$, but one can hardly say that $\ff{X}$ is well understood. As an illustration, the reader may take a moment to verify that
\begin{equation}\label{eq: rational identity}
\big(x-z\inv\big)\big(1-yx\big)\inv\big(y-z\big)  + \big(y\inv-z\inv\big)\big(1-x\inv y\inv \big)\inv\big(x\inv-z\big) = 0
\end{equation}
without allowing the variables to commute. A commonly used realization of the free field, due to Lewin \cite{Lew:1974}, involves Malcev--Neumann series over the free group $\fgx{X}$ generated by $X$. These are noncommutative, multivariate analogs of Laurent series; we recall the details in Section \ref{sec: prelims}. 

While the Lewin realization of the free field is powerful, it suffers from an inconvenient asymmetry. We illustrate with an identity of Euler \cite{Car:2000}:
\[
\sum_{k\in\mathbb{Z}} x^k = 0.
\]
This looks absurd, but rewriting it as $\sum_{k\geq 0} {x}^{-k} \,+\, x\sum_{k\geq 0} x^k$, we recognize the geometric series expansions of $1/(1-x\inv)$ and $x/(1-x)$, respectively. Adding these fractions together does indeed give zero. Of course, $\K\llbracket \ixx \rrbracket$ is not a ring, so the mixing of series in $x$ and $x\inv$ is usually disallowed. Nevertheless, passing through steps such as this one is a useful technique for proving identities in the free field. 
(See Appendix \ref{sec: Euler}.)

In this paper, we give Euler's identity firm footing in the noncommutative setting. (See \cite{BecHaaSot:2009,BrlReu:2003} for more on the commutative setting.) Our main result in this vein is a characterization of which series over the free group $\fgx{X}$ have rational expressions (Theorem \ref{th: new_connes_criterion} of Section \ref{sec: rank}).
The main ingredient in the proof is a Fredholm operator $\F$ on the free group that we describe in Section \ref{sec: connes}. This operator featured prominently in a conjecture of Alain Connes \cite{Con:1994} that was proven in \cite{DucReu:1997}. Our result extends the main result there. Briefly, \emph{an element $a\in \ser{\fg}$ is rational if and only if its associated Connes operator $\FF{a}$ has finite rank. }

In this paper, we also tackle the problem of rendering a given rational expression into Malcev--Neumann form (Section \ref{sec: implications}). Our main result here, Theorem \ref{th: word problem}, is an effective algorithm for solving the word problem in $\ff{X}$. Recall that the word problem for free groups (\emph{when are two words equal in $\fgx{X}$?}) is decidable, even though this is famously not the case for all groups \cite{Nov:1955,Boo:1958}. The analogous problem for the free field (\emph{when are two expressions equal in $\ff{X}$?}) was first considered in \cite{Coh:1973a} and taken up again in \cite{CR:1999}. 
Here is a simple example: 
\begin{center}
	\emph{Are \ $x\inv\,(1-x)\inv$ and \ $x\inv + (1-x)\inv$ equal?} 
\end{center}
The answer is, ``yes,'' as the reader may easily verify. 
While algorithms are presented in \cite{Coh:1973a,CR:1999}, their complexity seems very high. We derive our algorithm from Fliess' proof \cite{Fli:1971} of the following fact: \emph{simplifications in rational expressions preserve rationality.} Here, at last, we may say that the word problem is certainly not undecidable. Though the complexity of our algorithm could certainly be improved.

\section{The Free Field, Rational Series, and the Connes Operator}\label{sec: prelims}

\subsection{The free field}\label{sec: free field}
The free field $\ff{X}$ is defined as the initial object in the category of \emph{epic $\K\langle X\rangle$ skew fields with specializations.} (See \cite[Ch. 4]{Coh:1995}, which also contains a realization in terms of \emph{full matrices} over $\K\langle X\rangle$.) In what follows, we use Lewin's realization in terms of formal series. 

\subsubsection{Operations on series over the free group}\label{sec: all series}
Let $\fg = \fgx{X}$ denote the free group generated by a finite set of noncommuting indeterminants $X$. Let $\poly\fg$ denote the vector space with basis $\fg$ and $\ser\fg$ denote the \emph{formal series} over $\fg$, i.e., functions $a\colon\fg\rightarrow \K$, which we may write as $\sum_{\w\in \fg} (a,\w)\, \w$ or $\sum_{\w\in \fg}a_\w\, \w$ as convenient. The \emph{support} $\suppx{a}$ of a series $a$ is the set of its nonzero coefficients, $\{\w\in \fg \mid (a,\w)\neq 0\}$. 

The sum $a+b$ of two series over $\fg$ is well-defined and defined by $(a+b,\w) = (a,\w) + (b,\w)$ for all $\w\in \fg$. 

The \demph{Cauchy product} $a\cdot b$ (or $ab$) of two series is well-defined if 
\begin{gather*}\label{eq:series multiplication}
(ab,\w) := \sum_{\substack{\alpha,\beta\in \fg \\ \alpha\beta=\w}} (a,\alpha)(b,\beta)
\end{gather*}
is a finite sum for all $\w\in \fg$. We will also have occasion to use the \demph{Hadamard product} $a \odot b$, defined by $(a \odot b, \w) := (a,\w)(b,\w)$ for all $\w\in\fg$.

We define now a unary operation called the {\em star operation}: given $a\in\ser\fg$, we let $a^\ast$ denote the sum $1+a+a^2 + \dotsb, $ if this is a well-defined element of $\ser\fg$; in other words, if for any $\omega \in \fg$, there are only finitely many tuples $(\omega_1,\ldots,\omega_n)$ such that 
$\omega=\omega_1\cdots\omega_n$ and 
$(a,\omega_1)\cdots(a,\omega_n)\neq 0$. In this case, it is the inverse of $1-a$ in $\ser\fg$ under the Cauchy product. 


\subsubsection{Rational series over the free group}\label{sec: ratseries_group}
Fix a total order $<$ on $\fg$ compatible with its group structure. 
(See Appendix \ref{sec: Lyndon} for an elementary example.)
%
The \demph{Malcev--Neumann series} (with respect to $<$) is the subset $\mn{\fg}\subseteq\ser{\fg}$ of series with well-ordered support. That is, $a\in\mn{\fg}$ if every nonempty set $A\subseteq\suppx{a}$ has a minimum element. 
One shows that $\mn\fg$ is closed under addition, multiplication and taking inverses, i.e., $\mn\fg$ is a skew field. 
See \cite[Ch. 2.4]{Coh:1995}, \cite[Ch. 13.2]{Pass:1985}, or \cite[Ch. IV.4]{Sak:2009}. 
The inverse of a Malcev--Neumann series $a$ is built in a geometric series--type manner: if $a=\oo{a} \oo{\w} + \sum_{\oo{\w}<\w} a_\w \w$, then
\begin{equation}\label{eq:a_inverse}
a\inv = {\Big[\big(1-\sum_{\oo{\w}<\w} \widetilde{a_\w}\, \w{\w_{\mathrm{o}}^{-1}} \big)\oo{a}\oo{\w} \Big]}^{-1}
	=a_{\mathrm{o}}^{-1} \w_{\mathrm{o}}^{-1} {\Big[\sum_{\oo{\w}<\w}\widetilde{a_\w}\, \w{\w_{\mathrm{o}}^{-1}}\Big]}^* ,
\end{equation}
with $\widetilde{a_\w} = -a_{\w}\oo{a}\inv$ and $\w_{\mathrm{o}}^{-1}$ being the minimum element of $\suppx{a\inv}$.

Lewin \cite{Lew:1974} showed that the free field $\ff{X}$ is isomorphic to the \demph{rational closure} of $\K\langle X\rangle$ in $\mn{\fg}$, i.e., the smallest subring of $\mn\fg$ containing $\K\langle X\rangle$ and closed under addition, multiplication and taking inverses. An alternative proof, using Cohn's realization of the free field in terms of full matrices, appears in \cite{Reu:1999}.

\subsection{Rational series over free monoids}\label{sec: ratseries_monoid}

The rational series in $\mn\fg$ introduced above are a particular case of {rational series} over an alphabet $A$. In this more general notion, one starts with polynomials on the free monoid $\poly A^*$ and builds expressions with $(+,\ \times,\ (\arg)^*)$. (Here, for an expression $p\in \poly A^*$ without constant term, $p^* := 1 + p + p^2 + \dotsb.$) 
Indeed, let $\overline{X} = \{\overline{x} \mid x\in X\}$ represent formal inverses for the elements of $X$, and put $A = \bxx$. Then \eqref{eq:a_inverse} allows us to replace the $(\arg)^{-1}$ operation over $\fg$ with the $(\arg)^{\ast}$ operation over $A$. We do so freely in what follows and exploit results from the latter theory that we collect here.

\subsubsection{Hankel rank}\label{sec: hankel rank}
Let $S=\sum_w (S,w)w$ be a series over $A^*$. Given a nonempty word $u\in A^*$, we define the new series 
$u \circc S \in \monser{A}$, a \demph{(right) translate} of $S$, by
\[
  u \circc S := \sum_{w \colon w=\oo{w}u} (S,w)\oo{w}.
\]
Letting $\varepsilon$ denote the empty word, put $\varepsilon \circc S := (S,\varepsilon)$ and extend bi-linearly to view $\circc$ as a map $\circc \colon \poly{A^*} \otimes \monser{A} \to \monser{A}.$ (See \cite[Ch. III.4]{Sak:2009} or \cite[Ch. 1.5]{BerReu:2011}, where the image is denoted $S u\inv$.)
The \demph{Hankel rank} of the series $S$, introduced in \cite{Fli:1974a}, is the rank of the operator $\circc S \colon \poly{A^*} \to \monser{A}$. We extend this construction to series over the free group $\fg$ in Section \ref{sec: reverse_implication}. 

The following is classical. 

\begin{theo}[Fliess]\label{th: Fliess}
A series $S\in\monser{A}$ over a free monoid $A^*$ is rational if and only if its right translates $\left\{f\circc S \mid f\in\poly{A^*} \right\}$ span a finite dimensional subspace of $\monser{A}$.
\end{theo}

An easy corollary is that the Hadamard product preserves rationality.
For a proof, see \cite[Th. 1.5.5]{BerReu:2011}. 

\begin{theo}[Sch\"utzenberger]\label{th: Hadamard}
If $S,T$ are rational series, then $S\odot T$ is rational.
\end{theo}

\subsubsection{Recognizable series}\label{sec: recog}
We recall some additional results on recognizable series that will be useful in Section \ref{sec: implications}. Fix a series $S=\sum_w (S,w)w \in \monser{A}$.

Let $L = \suppx{S}$. The \demph{language} $L$ is said to be \demph{recognizable} if there is an automaton\footnote{A finite state machine that processes words and transitions between states by reading letters one at a time \cite{Sak:2009}.} that accepts $L$ and no other words. For example, the machine in Figure \ref{fig: recog} accepts $\supp (b+ac)^*$; the input state is marked with an arrow and the output state is doubly circled.
\begin{figure}[!hbt]
  \centering
    \unitlength=3.25pt
    \begin{picture}(25, 13)(0,0)
    \gasset{Nw=5,Nh=5,Nmr=2.5,curvedepth=3,loopdiam=6}
    \thinlines
    \node[Nmarks=i,iangle=90](A1)(0,7){{\small\sf0}}
    \rmark[rdist=.4](A1)
    \drawloop[loopangle=180](A1){$b$}
    \node(A2)(25,7){{\small\sf1}}
    \drawedge(A1,A2){$a$}
    \drawedge(A2,A1){$c$}
    \end{picture}
\caption{An automaton recognizing the language $\supp (b+ac)^*$.}
\label{fig: recog}
\end{figure}
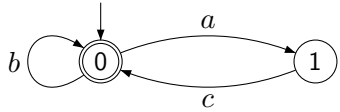
More generally, we may extend the notion of automata so that edge-labels carry coefficients other than 0 or 1, e.g., taking values in some semiring $\K$, and speak about $S$ being recognizable by a $\K$-automaton.

\begin{theo}[Sch\"utzenberger] A series $S$ is rational iff it is recognizable.
\end{theo}

We say that $S$ is \demph{representable}, with order $n$, if there exist row and column vectors $\lambda\in \K^{1\times n}$, $\rho \in \K^{n\times 1}$, and a monoid morphism $\mu\colon A^* \to \K^{n\times n}$ satisfying $(S,w) = \lambda \mu(w) \rho$ for all $w\in A^*$. For example, a representation $(\lambda,\mu,\rho)$ of the series $(b+ac)^*$ is provided by
\[
	\lambda = \begin{pmatrix} 1 & 0 \end{pmatrix} \!,	
	\quad \ \ 
	\mu(a) = \begin{pmatrix} 0 & 1\\ 0 & 0\end{pmatrix} \!,
	\ \ 
	\mu(b) = \begin{pmatrix} 1 & 0\\ 0 & 0\end{pmatrix} \!,
	\ \ 
	\mu(c) = \begin{pmatrix} 0 & 0\\ 1 & 0\end{pmatrix} \!,
	\quad \ \ 
	\rho = \begin{pmatrix}1\\ 0\end{pmatrix} \!.
\]
Here, $\mu(x)_{ij}=1$ iff there is an edge $i\xrightarrow{x}j$ in the automaton.
The equivalence between recognizable and representable series is well-known. (In fact, the usage of ``recognizable'' in the literature for what we call here ``representable'' is commonplace.) In \cite{Fli:1974a}, Fliess shows that the least $n$ for which $S$ has a representation is the Hankel rank of $S$ \cite[Th. 2.1.6]{BerReu:2011}.

\subsection{The Connes operator}\label{sec: connes}

In his book \emph{Noncommutative Geometry,} Connes \cite{Con:1994} gives a new proof of a celebrated $C^{\ast}$-algebra result%
\footnote{Namely, \emph{the reduced $C^{\ast}$-algebra of the free group does not contain nontrivial idempotents.} 
} 
of Pimsner--Voiculescu \cite{PimVoi:1982} using the machinery of ``Fredholm modules.'' We  recount the details that are relevant to the present work.

\subsubsection{Reduced words}\label{sec: reduced words}
Recall that each element $\w\in\fg$ has a unique \emph{reduced expression} $\w=x_1\dotsb x_n$ with $x_i\in X$ or $x_i\inv\in X$ for all $i$ and no pair $(i,i{+}1)$ satisfying $x_ix_{i+1} = 1_\fg$. We say that $\w$ has \demph{length} $\ell(\w)=n$. 

We identify ${\fg}$ with the subset of \demph{reduced words}, $\red\XX$, within the free monoid: say  $\w\in\XX$ is reduced if it contains no factor of the form $\sbar{x}x$ or $x\sbar{x}$ for any $x\in X$. We often denote the inverse of an element $\w\in\fg$ by $\overline{\w}$, and put $\overline{\overline{\w}}=\w$, to reduce clutter in the expressions that follow. 
Finally, if $\w_1, \dotsc, \w_r$ are elements of $\vv$, we write $\w \doteq \w_1\dotsb\w_n$ if the product is a \demph{reduced factorization} of $\w$ (in particular, each $\w_i$ is reduced). That is, $\w=\w_1\dotsb \w_r$ in $\fg$ and $\ell(\w) = \ell(\w_1) + \dotsb + \ell(\w_r)$.

Let $\w = x_1\dotsb x_n$ ($x_i\in X$ or $x_i^{-1} \in X$) be the reduced expression for $\w\in\fg\setminus\{1\}$. Its \demph{longest proper prefix} $\pp{\w}$ is the word $x_1\dotsb x_{n-1}$. Its \demph{(nonempty) suffixes} comprise the set $\suff{\w} = \{x_n,\,x_{n-1}x_n,\,\dotsc,\,x_1\dotsb x_{n-1}x_n\}$. 
We take the codomains of $\pp{\arg}$ and $\suff{\arg}$ to be $\fg$ or $\red\XX,$ as convenient, with context dictating which is intended.

\subsubsection{The Connes operator}\label{sec: operator}

Let $\cc = (\mathcal{V},\ee)$ be the Cayley graph of $\fg$, the infinite tree with vertex set $\mathcal{V}=\vv$ and edge set $\ee$ satisfying $\{\uu,\w\} \in \ee$ if and only if $\uu = \pp \w$ or $\w=\pp \uu$. (In this case, note that $\{\alpha\uu,\alpha\w\}$ is again an element of $\ee$, for any $\alpha\in\vv$.)
On the linear space $\poly\cc=\poly\vv\oplus\poly\ee$, define a bilinear operation $\cdot\colon \ser\fg \otimes \poly\cc \to \ser\cc$ as follows. For $a=\sum_{\alpha\in\fg} (a,\alpha) \alpha %
 \in \ser\fg$ and $\{\uu,\w\}\in\ee$, put
\begin{gather*}
  a\cdot \w := \sum_{\alpha\in\fg} (a,\alpha)\,\alpha\w 
	\quad\hbox{ and }\quad
  a\cdot\{\uu,\w\} := \sum_{\alpha\in\fg} (a,\alpha)\, \{\alpha\uu, \alpha\w\} .
\end{gather*}

Recall that a \demph{Fredholm operator} between two normed vector spaces $X$ and $Y$ is a bounded linear map with finite dimensional kernel and cokernel. The Fredholm operator $\F \colon \ser\cc \to \ser\cc$ that Pimsner--Voiculescu and Connes use is defined as follows. For all $\uu \in \mathcal{V}$ and $\{\pp\w,\w\} \in \ee$,
\begin{gather*}
  \F(\{\pp \w, \w\}) := \w 
	\qquad\hbox{ and }\qquad
  \F(\uu) := \left\{\begin{array}{@{\,}ll}
	0 & \hbox{if }\uu=1 \\
	\{\pp \uu, \uu\} & \hbox{otherwise.} 
	\end{array}\right. 
\end{gather*}
The reader may check that $\ker \F = \mathrm{co}{\ker \F} = \mathrm{span}_{\K} \{1\}$.

\begin{defi}
Given $a\in\ser\fg$, the \demph{Connes operator} $\FF{a} \colon \poly\cc \rightarrow \ser{\cc}$ is defined as the commutator $\F\circ a - a\circ\F $.
\end{defi}

We develop a formula for the action of $\FF{a}$ on an edge $\{\pp\w, \w\}$:
\begin{align}
\notag \FF{a}\,\big\{\pp{\w}, \w\big\} &= \F \circ a\circ \big\{\pp{\w}, \w\big\} - a\circ \F\circ \big\{\pp{\w}, \w\big\} \\
\notag &= \F\biggl( \sum_{\alpha\in\vv} a_\alpha\, \big\{\alpha\pp{\w}, \alpha\w\big\}\biggr) - a\cdot \w \\
\notag &= \Biggl( \sum_{\substack{\alpha\in\vv\rule[-2pt]{0pt}{4pt} \\ \overline{\w}\in\suff{\alpha}}} \!\! a_\alpha\, \alpha\pp{\w} + \sum_{\substack{\alpha\in\vv\rule[-2pt]{0pt}{4pt} \\ \overline{\w}\not\in\suff{\alpha}}} \!\! a_\alpha\, \alpha\w \Biggr) - 
	\Biggl(\sum_{\substack{\alpha\in\vv\rule[-2pt]{0pt}{4pt} \\ \overline{\w}\in\suff{\alpha}}} \!\! a_\alpha \alpha\w + 
	\sum_{\substack{\alpha\in\vv\rule[-2pt]{0pt}{4pt} \\ \overline{\w}\notin\suff{\alpha}}} \!\! a_\alpha \alpha\w \Biggr) \\
\label{eq: Fa formula} 
&= \sum_{\substack{\alpha\in\vv \rule[-2pt]{0pt}{4pt} \\ \overline{\w}\in\suff{\alpha}}} \!\! a_\alpha\, \alpha \bigl(\pp\w - \w \bigr) .
\end{align}

\begin{exam} $\FF{xxy}\{\bar y,\bar{y}\bar{x}\} = xxy(\bar{y}-\bar y\bar x) = xx-x$, while $\FF{xyy}\{\bar y, \bar y\bar x\} = 0$.
\end{exam}

After his proof of the Pimsner--Voiculescu result, Connes offers a conjecture about the rank of his operator \cite[p. 342, Remark 3]{Con:1994}. This was later proved by Duchamp and the second author \cite{DucReu:1997}. A discrete topology analog of the conjecture goes as follows.

\begin{theo}\label{th: connes_criterion} 
An element $a$ of the Malcev--Neumann series $\mn\fg$ belongs to the rational closure of $\poly\fg$ in $\mn\fg$ $\iff$ the operator $\FF{a}$ is of finite rank. That is, if the image $\FF{a}\poly\cc \subseteq \ser\cc$ is finite dimensional. 
\end{theo}

This appears as Th{\'e}or{\`e}me 12 in \cite{DucReu:1997}, but the proof of the forward implication is not given. A complete proof is given here, as it follows from the proof of our first result below.

\section{Rank of the Connes Operators}\label{sec: rank}

By the {\em rational closure} of $\poly\fg$ in $\ser\fg$, we mean the smallest subspace of $\ser\fg$ that contains $\poly\fg$ and is closed under the Cauchy product and the star operation. Here, we insist only that all products and stars involved are well-defined in the sense of Section \ref{sec: all series}. Our first result is the following generalization of Theorem \ref{th: connes_criterion}. 

\begin{theo}\label{th: new_connes_criterion} 
A formal series $a\in\ser\fg$ belongs to the rational closure of $\poly\fg$ in $\ser\fg$ $\iff$ the operator $\FF{a}$ is of finite rank. 
\end{theo}

Sections \ref{sec: F-identities} and \ref{sec: closed} establish the necessary machinery for the proof. Sections \ref{sec: forward_implication} and \ref{sec: reverse_implication} establish the forward and reverse implications, respectively. 
(See Appendix \ref{sec: Extensions} for possible refinements and extensions.)

As a first simplification, note that $\FF{\arg} \colon \poly\cc \to \ser\cc$ has finite rank if and only if the restricted operator  $\FF{\arg}\big\vert_{\poly\ee} \colon \poly\ee \to \ser\fg$ does. We focus on $\poly\ee$ in what follows.

\subsection{Elementary identities on Connes operators}\label{sec: F-identities}

Given $a,b\in\ser\fg$, it will be useful to express the operators $\FF{ab}$ and $\FF{a^{\ast}}$ in terms of $\FF{a}$ and $\FF{b}$. 

\begin{prop}\label{th: connes identities} 
Suppose $a,b\in\ser\vv$ are such that $a b$ and $a^{\ast}$ are well-defined. Then $\FF{ab}$ and $\FF{a^{\ast}}$ are well-defined operators; moreover
\begin{eqnarray}
\label{eq: Fab} \FF{ab} &=& \FF{a}b + a\FF{b} ,\\[1ex]
\label{eq: Fa*} \FF{a^{\ast}} &=& a^{\ast}\FF{a}a^{\ast} .
\end{eqnarray}
\end{prop}

\begin{proof}[Proof of Identity \eqref{eq: Fab}] 
First, some seemingly inoccuous algebra:
\[
\FF{ab} = \F ab - ab\F = (\F ab - a\F b) + (a\F b - ab\F) = \FF{a}b + a\FF{b}.
\]
Supposing $a,b,ab\in\ser\fg$, the operators $\FF{a},\, \FF{b},\, \FF{ab}\colon \poly\ee\rightarrow\ser\vv$ are well-defined, in particular $\F ab$ and $ab\F$ are well-defined operators by themselves. In order to conclude that the left- and right-hand sides of \eqref{eq: Fab} are equal, we must check that $a\F b$ is also well-defined as an operator. This is handled in the next lemma.
\end{proof}

\begin{lemm}\label{th: keylemma} If $a,b,$ and $ab$ belong to $\ser\fg$, then $a\F b$ is a well-defined operator from $\poly\ee$ to $\ser\fg$.
\end{lemm}

\begin{proof} Write $a=\sum_{\alpha\in\fg} a_\alpha \alpha$ and $b=\sum_{\beta\in\fg} b_\beta \beta$. Since $ab\in\ser\fg$, we have
\begin{gather}\label{eq: c=ab}
ab = \sum_{\gamma\in\fg} \Biggl(\,\sum_{\substack{\alpha,\beta\in\fg \\ \gamma=\alpha\beta}} \!a_\alpha\, b_\beta\Biggr) \gamma ,
\end{gather}
with $\sum_{\gamma=\alpha\beta} a_\alpha\, b_\beta$ involving finitely many nonzero terms for each $\gamma\in\fg$. The effect of $a\F b$ on $\big\{\pp{\omega},\omega\big\}\in\ee$ is as follows:
\begin{align}
\notag a\F b\,\big\{\pp{\omega},\omega\big\} &= a\F\sum_{\beta\in\vv} b_\beta\,  \big\{\beta\pp{\omega},\beta\omega\big\} \\
\notag &= a\Biggl(\,\sum_{\substack{\beta\in\vv\rule[-2pt]{0pt}{4pt} \\ \overline{\omega}\in\suff{\beta}}} \!b_\beta\, \beta\pp{\omega} + \sum_{\substack{\beta\in\vv\rule[-2pt]{0pt}{4pt} \\ \overline{\omega}\not\in\suff{\beta}}} \!b_\beta\, \beta\omega \Biggr) \\
\label{eq: aPb} &= \sum_{\substack{\alpha,\beta\in\vv\rule[-2pt]{0pt}{4pt} \\ \overline{\omega}\in\suff{\beta}}} \!a_\alpha b_\beta\, \alpha\beta\pp{\omega} + \sum_{\substack{\alpha,\beta\in\vv\rule[-2pt]{0pt}{4pt} \\ \overline{\omega}\not\in\suff{\beta}}} \!a_\alpha b_\beta\, \alpha\beta\omega  \,.
\end{align}
We claim that no $\uu\in\vv$ appears infinitely often in \eqref{eq: aPb}. Indeed, suppose $\{(\alpha_{i},\beta_{i})\}$ is an infinite sequence satisfying $\upsilon\in\big\{\alpha_{i\,}\beta_{i\,}\pp{\omega},\,\alpha_{i\,}\beta_{i\,}\omega \big\}$ for all $i$. Then there is an infinite subsequence $\{(\alpha_{j},\beta_{j})\}$ satisfying $\forall j$, $\uu = \alpha_{j\,}\beta_{j\,}\pp\omega$ or $\forall j$, $\uu = \alpha_{j\,}\beta_{j\,}{\omega}.$ The two cases are similar; we consider the latter. 
%
%
Let $\uu'{\omega}$ denote the common value of the elements of the subsequence, so $\alpha_j\beta_j=\uu'$ ($\forall j$). Deduce from \eqref{eq: c=ab} that only a finite number of the $\alpha_{j\,}\beta_{j}$ appear with nonzero coefficient $a_{\alpha_j}b_{\beta_j}$. This proves the claim and the Lemma.
\end{proof}


\begin{proof}[Proof of Identity \eqref{eq: Fa*}] 
If $a\in\ser\fg$ is such that $a^{\ast}$ is well-defined, then $a^{\ast} a$, $a a^{\ast}$, $a^{\ast}a^{\ast}$ and $a^{\ast} a a^{\ast}$ are also well-defined, and $(1-a)a^{\ast} = a^{\ast}(1-a) = 1$. Thus we may write
\begin{eqnarray*}
\FF{a^{\ast}} &=& \F a^{\ast} - a^{\ast}\F = a^{\ast}(1-a)\F a^{\ast} - a^{\ast}\F(1-a)a^{\ast} \\
&=& a^{\ast}\F a^{\ast} - a^{\ast}a\F a^{\ast} - a^{\ast}\F a^{\ast} + a^{\ast}\F a a^{\ast} = a^{\ast}\FF{a}a^{\ast}.
\end{eqnarray*}
The above observations and Lemma \ref{th: keylemma} guarantee that the operators $a^{\ast}\F a^{\ast}$, $a^{\ast}a\F a^{\ast}$ and $a^{\ast}\F a a^{\ast}$ used in the intermediate steps are well-defined. Conclude that \eqref{eq: Fa*} holds whenever $a^{\ast}$ is well-defined.
\end{proof}

\subsection{Closed subspaces of $\ser{\vv}$}\label{sec: closed}
The following is a standard result on topological vector spaces. We include a proof for the sake of completeness.

\begin{lemm}\label{th: closed} Suppose $V$ is a finite dimensional subspace of $\ser\vv$. Then $V$ is closed with respect to the product topology on $\ser\vv$ extending the discrete topology on $\K$.
\end{lemm}
\begin{proof} In five easy steps.
\smallskip 

\textit{(i).\,} Recall that in the discrete topology on $\K$, a sequence converges if and only if it is eventually constant.
\smallskip 

\textit{(ii).\,} Fix $a\in \ser{\fg}$ and write $a = \sum_\alpha (a,\alpha)\,\alpha$. The product topology on $\ser\vv$ is such that a sequence $\{a_n \in \ser\vv\} \subseteq V$ converges to $a$ if and only if for each $\alpha\in\vv$, there exists $N_\alpha$ satisfying $(a_n,\alpha)=(a,\alpha)$ for all $n>N_\alpha$. 
\smallskip 

\textit{(iii).\,} Suppose that $b_1,\dotsc, b_m$ is a basis for the finite dimensional subspace $V\subseteq\ser\vv$. Find elements $\alpha_1,\dotsc, \alpha_m \in\vv$ satisfying
\begin{gather*}
\det\! \begin{pmatrix}
(b_1,\alpha_1) & (b_2,\alpha_1) & \dotsb & (b_m,\alpha_1) \\
(b_1,\alpha_2) & (b_2,\alpha_2) & \dotsb & (b_m,\alpha_2) \\
\vdots & \vdots & \ddots & \vdots \\
(b_1,\alpha_m) & (b_2,\alpha_m) & \dotsb & (b_m,\alpha_m) 
\end{pmatrix} \neq 0,
\end{gather*}
which are guaranteed to exist by the independence of $b_1,\dotsc,b_m$. 
\smallskip 

\textit{(iv).\,} Given $\{a_n\}\to a$ as above, find constants $s_i^{(n)}\in \K$ satisfying
\begin{equation}\label{eq: an=sum}
a_n = s_1^{(n)} b_1 + s_2^{(n)} b_2 + \dotsb + s_m^{(n)} b_m \quad (\forall n).
\end{equation}
Put $N=\max\left\{N_{\alpha_1},\dotsc,N_{\alpha_m}\right\}$, with $N_{\alpha_i}$ determined as in (ii). Let $B$ be the matrix found in (iii). From \eqref{eq: an=sum}, we have the system 
\begin{gather*}
\begin{pmatrix}
(b_1,\alpha_1) & \dotsb & (b_m,\alpha_1) \\
\vdots & \ddots & \vdots \\
(b_1,\alpha_m) & \dotsb & (b_m,\alpha_m) 
\end{pmatrix}\!\begin{pmatrix}
s_1 ^{(n)} \\
\vdots \\
s_m^{(n)}
\end{pmatrix} = \begin{pmatrix}
(a_n,\alpha_1) \\
\vdots \\
(a_n,\alpha_m)
\end{pmatrix}.
\end{gather*}
Note that the right-hand side is constant for $n>N$. Specifically, it equals $[(a,\alpha_1) \ \dotsb \ (a,\alpha_m)]^T$. Since $B$ is invertible, this system has a unique solution $[s_1 \ \dotsb \ s_m]^T$ independent of $n$ (for $n > N$). 
\smallskip 

\textit{(v).\,} We claim that
\[
	a= s_1 b_1 + s_2 b_2 + \dotsb + s_m b_m \,.
\]
To see this, note that \eqref{eq: an=sum} reduces to 
$
	a_n= s_1 b_1 + \dotsb + s_m b_m
$ 
for $n>N$, since the $b_i$ are linearly independent. It follows that $a_n=a$ for all $n>N$ and the lemma is proven.
\end{proof}

\subsection{Connes operators of rational elements}\label{sec: forward_implication}
Given the preparatory results in Sections \ref{sec: F-identities} and \ref{sec: closed}, we are ready to prove that Connes operators of rational elements have finite rank.

\begin{proof}[Proof of Theorem \ref{th: new_connes_criterion} (Forward~Implication)]
It is enough to induct on the $(+,\times,\ast)$-complexity of a rational series $a\in\ser\fg$. We check that if $a\in\poly\fg$ then $\FF{a}$ has finite rank and that the finite rank condition is closed under $+$, $\times$ and $\ast$. 
\smallskip

\par
\emph{(i):\, Closed under $+$.\,} Suppose $a=a' + a''$ with $\FF{a'}$ and $\FF{a''}$ finite rank. Then $\img\,\FF{a}\subseteq \img\,\FF{a'} + \img\,\FF{a''}$, thus $\FF{a}$ has finite rank as well.
\smallskip 

\emph{(ii):\, Polynomials have finite rank.\,} After (i) we may assume $a$ is a monomial $\alpha\in\fg$. Apply $\FF{a}$ to some $e=\sum_{\omega} e_{\omega}\,\big\{\pp{\omega},\omega\big\}$ to get
\[
\sum_{\omega \colon \overline{\omega}\in\suff{\alpha}} e_\omega \alpha \bigl(\pp\omega - \omega\bigr) .
\]
Conclude that the rank of $\FF{a}$ is precisely $\ell(\alpha)$, the cardinality of $\suff{\alpha}$.
\smallskip 

\emph{(iii):\, Closed under $\times$.\,} Given rational series $a,b$ with $\FF{a},\FF{b}$ finite rank, consider $\FF{ab}$. By \eqref{eq: Fab} we have
\[
\img\FF{ab} \subseteq \img a\FF{b} + \img\FF{a}b.
\]
Since $\FF{b}$ has finite rank, we know that $\img a\FF{b}$ is finite dimensional. To control $\img \FF{a}b$, we use the fact that $\img\FF{a}$ is a closed subspace (Lemma \ref{th: closed}). First, approximate $b=\sum_\beta b_\beta \, \beta$ by the sequence 
\[
b_n := \sum_{\beta\colon \ell(\beta)\leq n} b_\beta\, \beta \,.
\]
Given $e\in\poly\ee$, note that $\{b_n e\} \rightarrow be$. This means $\FF{a}b_n e \rightarrow \FF{a}be$ and $\FF{a}be$ belongs to $\img\FF{a}$ 
(since this space is closed). Conclude that $\img\FF{a}b$ is finite dimensional.
\smallskip 

\emph{(iv):\, Closed under $\ast$.\,} Suppose $a\in\ser\fg$ with $a^{\ast}$ well-defined and $\FF{a}$ having finite rank. Use \eqref{eq: Fa*} to write $\FF{a^{\ast}} = a^{\ast}\FF{a}a^{\ast}$. The two arguments in (iii) may be combined to conclude that $a^{\ast}\FF{a}a^{\ast}$ has finite rank.
\end{proof}

\subsection{Connes operators with finite rank}\label{sec: reverse_implication}
We first reprise Proposition 5 in \cite{DucReu:1997} to bound the rank of $\FF{a}$ when $a$ is rational, making the necessary changes to work over $\ser\fg$ instead of $\mn\fg$. The completed proof of Theorem \ref{th: new_connes_criterion} then follows.

Given $\w\in\vv\setminus\{1\}$ and $a=\sum_{\alpha \in\fg} a_\alpha \, \alpha \in\ser\fg$, we define the series $\w \circc a\in\ser\vv$ by
\begin{gather}\label{eq: Hankel}
	\w\circc a  := \sum_{%
	\alpha\doteq\oo\alpha\w} a_\alpha \oo\alpha 
\end{gather}
and extend this to define a bilinear map $\circc \colon \poly(\vv\setminus \{1\}) \otimes \ser\fg \to \ser\fg$ in the natural way. (The symbol $\doteq$ has been defined in Section \ref{sec: reduced words}.) We now compare the {Connes rank} of $\FF{a}$ on $\poly\ee$ to the {Hankel rank} of ${\circc a}$ on $\poly(\vv\setminus\{1\})$.

\begin{prop}\label{th:rank bounds} Fix $|X|=n$. If $a\in\ser\vv$ belongs to the rational closure of $\poly\fg$ in $\ser\vv$ then the Connes and Hankel ranks of $a$ are related by
\begin{gather*}
\frac{1}{2n}\, \rnk{\circc a} \leq \rnk {\FF{a}} \leq 2n\,\rnk{\circc a} \,.
\end{gather*}
\end{prop}

\begin{rema} The upper bound on $\rnk \FF{a}$ is tight. Consider $X=\{x,y\}$ and the series $a=x+y+\overline{x} + \overline{y}$. 
\end{rema}

\begin{proof} Fix $\omega=\pp\omega x\in\fg$ and $a=\sum a_\alpha\, \alpha \in\ser\fg$. The proof rests on the observation that
\begin{align}\label{eq: auxtrans}
	\FF{a}\{\pp\w,\w\} = (\overline{\w} \circc a)\cdot(\overline{x}-1) .
\end{align}
%
(Compare \eqref{eq: Fa formula}.) To demonstrate the second inequality, we define auxillary linear transformations $\auxtrans{a}{\F}{y} \colon \poly\ee \rightarrow \ser\vv$ for each $y\in X\cup \overline X$ by 
\[
\{\pp\w, \w\}  \mapsto \bigl(\overline{\w}\circc a\bigr)\cdot(y-1) = \sum_{%
\alpha\doteq\oo\alpha\overline\w} a_\alpha\,\oo\alpha (y-1).
\]
From \eqref{eq: Hankel} we have $\rnk \auxtrans{a}{\F}{y} \leq \rnk{\circc a}$, and from \eqref{eq: auxtrans} we have $\img \FF{a} \subseteq \sum_{y} \img \auxtrans{a}{\F}{y}$. So indeed, 
$
\rnk \FF{a} \leq 2n\, \rnk{\circc a}.
$
\smallskip

To demonstrate the first inequality, we build (partially defined) operators $\revtrans{a}{\F}{y} \colon \poly\ee \rightarrow \ser\vv$ for each $y\in \bxx$ as follows:
\[
\{\pp\w,\w\} \mapsto \Bigl(\FF{a}\{\pp\w,\w\}\Bigr)\cdot(-y^*) = -\Biggl[\sum_{\alpha \colon \overline{\omega}\in\suff{\alpha}} a_\alpha\, \alpha\,\bigl(\pp\w-\w\bigr) \Biggr]y^{\ast} \,,
\]
with $y^* = 1 + y + y^2 + \dotsb$, as usual. The set on which any $\revtrans{a}{\F}{y}$ is defined is a linear subspace of $\K\ee$, and it acts linearly there. We claim that $\revtrans{a}{\F}{\overline x} \{\pp\w,\w\}$ is well-defined. If so, we are done, because
\begin{align*}
\revtrans{a}{\F}{\overline x} \{\pp\w,\w\} &= -\Biggl[\sum_{%
\alpha\colon \overline\w\in\suff{\alpha}} a_\alpha\,\alpha \bigl(\pp\w -\w\bigr) \Biggr]\overline{x}^{\,\ast} 
=\Bigg[\sum_{\alpha\doteq\oo\alpha\overline\w} a_\alpha\,\oo\alpha \Bigg]\Bigg[\left( 1-\sbar{x} \right) \sbar{x}^{\,\ast}\Bigg] \\
&= \Bigg[\sum_{\alpha\doteq\oo\alpha\overline\w} a_\alpha\,\oo\alpha \Bigg] = \overline{\omega}\circc a .
\end{align*}
In other words, $\img{\circc a} \subseteq \sum_y \img \revtrans{a}{\F}{y}$, from which it follows that 
$
\rnk{\circc a} \leq 2n\, \rnk \revtrans{a}{\F}{\overline x} \leq 2n\, \rnk \FF{a} .
$

Toward verifying the claim, we write $\FF{a}\,\{\pp\w,\w\}$ in standard series form, namely
\[
\FF{a}\,\{\pp\w,\w\} = \sum_{\alpha \colon x\not\in\suff{\alpha}} \left(a_{\alpha\overline{\omega}} - a_{\alpha\sbar{x}\overline{\omega}} \right)\alpha\sbar{x} .
\]
The expression we must analyze is
\[
\revtrans{a}{\F}{\overline x}\{\pp\w,\w\} = \sum_{\substack{l\in\N,\,\alpha\in\fg \\ x\not\in\suff{\alpha}}} \left(a_{\alpha\overline{\omega}} - a_{\alpha\sbar{x}\overline{\omega}} \right)\alpha\sbar{x}\overline{x}^{\,l} .
\]
Fix $\uu\in\fg$ and suppose $\left\{(\alpha_i,l_i) : i\in I\right\}$ is a sequence satisfying $\alpha_i \sbar{x}\overline{x}^{\,l_i} = \uu$. For each $l$ appearing in the sequence $(l_i)$, there is a unique allowable $\alpha_i$, so we may order $\left\{(\alpha_i,l_i) : i\in I\right\}$ by its second factor, saying $(\alpha_i,l_i) < (\alpha_j,l_j)$ whenever $l_i<l_j$. Put $l_1$ equal to the minimum $l$ occuring. Since each $\alpha_i$ satisfies $x\not\in\suff{\alpha_i}$, the equalities $\uu = \alpha_i\sbar{x} \overline{x}^{\,l_i}$ are all reduced factorizations of $\uu$. Finally, the chain
\[
\alpha_1\sbar{x}\overline{x}^{\,l_1} = \alpha_2\sbar{x}\overline{x}^{\,l_2} = \dotsb \quad (l_1<l_2<\dotsb)
\]
forces $\alpha_2$ to be a proper prefix of $\alpha_1$, and forces $\alpha_3$ to be a proper prefix of $\alpha_2$, and so on. Since the set of prefixes of $\uu$ is finite, $I=I(\uu)$ is a finite set and $\uu$ does not occur infinitely often in $\revtrans{a}{\F}{\overline x}$.
\end{proof}

\begin{proof}[Proof of Theorem \ref{th: new_connes_criterion} (Reverse Implication)] 
Identify $\fg$ with $\red\XX$, then use Theorem \ref{th: Fliess} and Proposition \ref{th:rank bounds}.
\end{proof}

\section{Applications to $\mn\fg$ and $\ff{X}$}
\label{sec: implications}

It is easy to see that our proof of Theorem \ref{th: new_connes_criterion} works equally for Malcev--Neumann series, thereby giving a complete proof of Theorem \ref{th: connes_criterion}. Here we discuss two further results stemming from the notion of rational series in $\ser\fg$. 

\subsection{Expressions without simplifications, effectively}\label{sec: reduced-effectively} 
Recall that an element $\omega = x_1 \dotsb x_n$ ($x_i \in X$ or $x_i^{-1} \in X$) in $\fg$ is called reduced if no pair $(i,i+1)$ satisfies $x_ix_{i+1}=1_\fg$. This is the starting point for our next notion: \emph{($*$-rational) expressions without simplification} in $\ser\fg$. A sum of monomials in $\poly\fg$ is an expression without simplification if each monomial involved is reduced. More generally, the sum $a+b$ of two expressions without simplification is again an expression without simplification. (We are only concerned with multiplicative simplifications, so $(xy\overline{x} + x) + (y^2 - x)$ is an expression without simplification.) Continuing, the product $a\cdot b$ of two expressions without simplification is an expression without simplification if for every monomial $\alpha$ in the support of $a$ and every monomial $\beta$ in the support of $b$, the monomial $\alpha\beta$ is reduced, i.e., $\ell(\alpha) + \ell(\beta) = \ell(\alpha\beta)$. If $a$ is an expression without simplification and $a^*$ is well-defined, then $a^*$ is an expression without simplification if and only if $a\cdot a$ is an expression without simplification. 

\begin{exam}\label{ex: simplifications} As a simple example, note that $\sbar{x}x^*$ may be written $\overline{x} + x^*$. For another example, consider $(\overline{x}yx)^*$, which is a legal expression $\ser\fg$ (and even in $\mn\fg$ if $1<y$ and $xy<yx$). On the one hand, it is equal to $\sum_{n\geq0} (\overline{x}yx)^n$, which has many simplifications; on the other hand, it is equal to $\overline{x}y^*x$, which has none.
\end{exam}

\subsubsection{Rational series in $\ser\fg$}
Given a well-defined $*$-rational series $a\in\ser\fg$, we know that the corresponding series of reduced words $S \in \monser{\bxx}$ is rational \cite{Ben:1969}. 
Our goal is an effective algorithm that yields a rational expression for $S$ without simplification. It suffices to make a careful analysis of Fliess' construction \cite[Th. 3]{Fli:1971}. (See \cite{Sak:2009} 
for details on $\K$-automata and Cartesian products.)

\begin{theo}\label{th: effective}
Given a well-defined $*$-rational expression, representing the series $a\in \ser{\fg}$, there exists an equivalent $*$-rational expression without simplification that is effectively computable.
\end{theo}

\begin{proof}
Given such a series $a$, we: first, produce a ($\K$-)automaton that recognizes $a$; next, modify the automaton so that all of its paths are reduced words; and finally, construct a rational expression equivalent to $a$ from the augmented automaton.
\smallskip

\emph{(i):\, Fliess' construction (simplified case).} Let $a$ be a well-defined $*$-rational series in $\ser{\fg}$, recognized by an automaton $\mathcal A$ over $\bxx$, with $n$ nodes. Suppose that $\mathcal A$ contains the transitions $p\xrightarrow{x} q \xrightarrow{\overline x} r$ for some $x\in\bxx$. For each transition $o\xrightarrow{y} p$ in $\mathcal A$ (with $y\in\bxx$), add the transition $o\xrightarrow{y} r$. 
(Note that we allow $o$ to be the incoming edge to the automaton, with label $1$, in which case we add another incoming edge $o\xrightarrow{} r$.) 
Repeat for all such triples $p\xrightarrow{x} q \xrightarrow{\overline x} r$ originally occurring in $\mathcal A$. This builds an augmented automaton $\mathcal A'$. Then repeat for any triple $p\xrightarrow{x} q \xrightarrow{\overline x} r$ in $\mathcal A'$ not already appearing in $\mathcal A$ to build an automaton $\mathcal A''$. Continue in this way until all new reductions have been bypassed.
We claim that this process terminates in a finite (not necessarily deterministic) automaton $\tilde{\mathcal A}$. Indeed, at each point, one constructs a subgraph of the complete, labeled, directed graph on $n$ nodes and $2|X|$ labels. This is a finite graph. 
\smallskip

\emph{(ii):\, Fliess' construction (generic case).} We modify the above construction to accommodate well-defined $*$-rational series $a \in \ser{\fg}$ with more general coefficients. First, build a trim $\K$-automaton $\mathcal A$ over $\bxx$, with $n$ nodes, following its presentation as a rational series. 

Each triple $o\xrightarrow{\alpha y} p \xrightarrow{\beta x} q \xrightarrow{\gamma \overline x} r$ ($x,y\in\bxx$; $\alpha,\beta,\gamma\in \K$) will be handled in roughly the same way as above; more precisely, augment $\mathcal A$ to an automaton $\mathcal A'$ by adding the edge $o\xrightarrow{(\alpha\beta\gamma) y} r$. The concern is that the procedure $\mathcal A',\mathcal A'',\ldots$ never terminates, or rather, results in an infinite automaton $\mathcal A^{(\infty)}$. This cannot happen, as we now argue. 

Evidently, at each step in the procedure, $\mathcal A^{(l)} \mapsto \mathcal A^{(l+1)}$, no new nodes are added, and any added edge comes from choosing a letter from $\bxx$ and a coefficient from $\K$. Conclude that if $\mathcal A^{(\infty)}$ is an infinite graph, then there are nodes $o,r$ in $\mathcal A^{(\infty)}$ and a letter $x\in \bxx$ so that $o \xrightarrow{} r$ has an infinite set of labels $\{\xi_1x, \xi_2x, \dotsc\}$, with $\xi_i\in \K$. Now, since $\mathcal A$ was trim, the rational expression for $a$ contains a coefficient computation $(a,\w)$ that is an infinite sum, contradicting the assumption that the expression was well-defined.
 
So we may assume that $\mathcal A^{(\infty)}$ is in fact a finite $k$-automaton $\tilde{\mathcal A}$, constructible in finitely many steps. 
\smallskip

\emph{(iii):\, A rational expression without simplifications.}
It is well-known that the reduced words $\red\XX$ form a rational set,\footnote{Combine, e.g., Proposition 6.2 and Theorem 6.1 in \cite[Ch. II]{Sak:2009}.} and thus $\chi := \sum_{w\in \red\XX} w$ is recognizable by an automaton $\mathcal X$. In particular, each path accepted by $\mathcal X$ has a reduced word as its label. Form the Cartesian product $\tilde{\mathcal A} \times \mathcal X$. This automaton recognizes the Hadamard product $S\odot \chi =\sum_{w\in \red\XX} (S,w)w$. That is, $S$ itself.
Moreover, each path in $\tilde{\mathcal A} \times \mathcal X$ has the property that its label is a reduced word. 
%
It follows that $S$, and hence $a$, has a rational expression without simplification. (For this last point, simply reverse the McNaughton--Yamada algorithm, which builds an automata from a given rational expression.\footnote{See \cite{McNY:1960} or the proof of Theorem 5.1 in \cite[Ch. VII]{Eil:1974}.}) 
\end{proof}

\subsubsection{Rational series in $\mn\fg$}
We may also deduce a Malcev--Neumann version of the above result.

\begin{lemm}\label{th: well-ordered}
Given any rational Malcev--Neumann series $a$ over $\fg$, there exists a $\ast$-rational subset $L$ of $\fg$ satisfying 
$\suppx{a} \subseteq L$ and 
$L$ is well-ordered.
\end{lemm}

\begin{proof}
The following facts are well-known \cite[Lem. 13.2.9]{Pass:1985}, \cite[Ch. 2.4]{Coh:1995}.

\begin{enumerate}
\item[$(+)$]  If $L_1$ and $L_2$ are well-ordered, then $L_1 \cup L_2$ is well-ordered.
\item[$(\times)$] If $L_1$ and $L_2$ are well-ordered, then $L_1 L_2$ is also.
\item[$(\,\ast\,)$] If each $\w\in L$ satisfies $\w\neq1$ and $L$ is well-ordered, then  $L^*$ is also.
\end{enumerate}
Now construct $L$ recursively, following the rational presentation of $a$. 
\end{proof}

\begin{coro}\label{th: MN effective}
Every $*$-rational Malcev--Neumann series over the free group has an equivalent $*$-rational expression without simplification that is effectively computable.
\end{coro}

\begin{proof}
Given a $*$-rational expression for $a\in\mn{\fg}$, let $L$ be as in Lemma \ref{th: well-ordered} and let $\chi_L$ be the corresponding rational series over $\red\XX$, well-ordered after its identification with $\fg$. Observe that in a trim automaton $\mathcal X$ for $\chi_L$, each closed path has a label $w>1$. Indeed, otherwise $L$ contains  $\suppx{u w^* v}$ for some $u,v \in\XX$, making $L$ not well-ordered. 
Now follow Parts (ii) and (iii) in the proof of Theorem \ref{th: effective}---replacing $\mathcal X$ by $\mathcal X_L$.
\end{proof}

\begin{rema}
Note that in the result above we assume that an expression for $a\in \mn{\fg}$ has been given with inverses computed according to \eqref{eq:a_inverse}. In particular, we needn't determine the minimal element in the support of any subexpression within $a$. (However, such decisions may also be made effectively, as we show in Proposition \ref{th: min-supp}.)
\end{rema}

\begin{exam}\label{ex: effective}
Suppose $x<yx\in\fg$, so that $(\overline xyx)^*\overline x \in\mn\fg$. Then
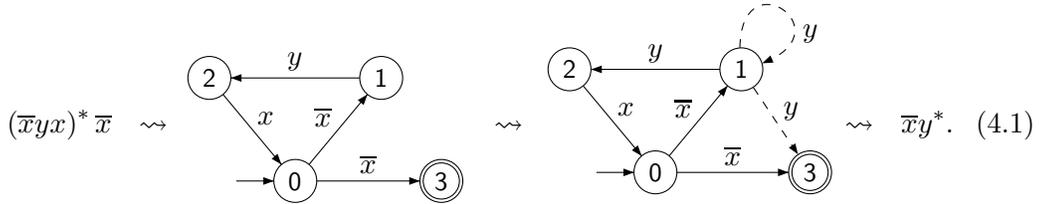
\begin{gather}
\left(\overline xyx\right)^*\overline x \ \   
\leadsto 
\!\!\!\!\raisebox{-28pt}{
	\psset{unit=3.25pt} 
    \begin{pspicture}(35, 18)(0,0)
    \gasset{Nw=5,Nh=5,Nmr=2.5,curvedepth=0,loopdiam=6}
    \thinlines
    \node[Nmarks=i,iangle=180](A0)(15,3){{\small\sf0}}
    \node(A1)(25,15){{\small\sf1}}
    \node(A2)(5,15){{\small\sf2}}
    \node(A3)(32,3){{\small\sf3}}
    \rmark[rdist=.4](A3)
    \drawedge(A0,A1){$\overline x$}
    \drawedge[ELside=r](A1,A2){$y$}
    \drawedge(A2,A0){$x$}
    \drawedge(A0,A3){$\overline  x$}
    \end{pspicture}
}
\ 
\leadsto 
\!\!\!\raisebox{-28pt}{
	\psset{unit=3.25pt} 
    \begin{pspicture}(35, 24)(0,-1)
    \gasset{Nw=5,Nh=5,Nmr=2.5,curvedepth=0,loopdiam=6}
    \thinlines
    \node[Nmarks=i,iangle=180](A0)(15,3){{\small\sf0}}
    \node(A1)(25,15){{\small\sf1}}
    \node(A2)(5,15){{\small\sf2}}
    \node(A3)(33,3){{\small\sf3}}
    \rmark[rdist=.4](A3)
    \drawedge(A0,A1){$\overline x$}
    \drawedge[ELside=r](A1,A2){$y$}
    \drawedge(A2,A0){$x$}
    \drawedge(A0,A3){$\overline  x$}
    \drawloop[loopangle=55,dash={1}0,ELpos=70](A1){$y$}
    \drawedge[dash={1}0](A1,A3){$y$}
    \end{pspicture}
}
\leadsto \ \  
\overline x y^* .
\end{gather}
Here, the dashed lines indicate those edges added during the Fliess construction. Explicit computation of the Hadamard product with $\mathcal X_L$ is suppressed.
\end{exam}

\subsection{The word problem in the free field}\label{sec: word problem} 
By the \demph{word problem} in the free field, we mean the following: 
\begin{center}
\emph{Given an expression for some $a \in \ff{X}$, determine if $a=0$.} 
\end{center}
This problem was first solved in \cite{Coh:1973a}, using Cohn's theory of full matrices. This solution was revisited in \cite{CR:1999}, where it was reformulated as an ideal-membership problem in a commutative ring (a solution is then possible using Gr{\"o}bner bases and Buchberger's algorithm).
Here, we describe an alternative solution to the word problem, debarking from the Malcev--Neumann realization of $\ff{X}$. 

\subsubsection{Description of the algorithm}

We need the following standard result from the theory of rational languages.

\begin{prop}\label{th: subwords}
Given a rational series $S\in \monser{A}$ of (Hankel) rank $n$, each word of length at least $n$ in the support has a subword which is also in the support. 
\end{prop}

\begin{proof}
Let $w\in\suppx{S}$ be of length at least $n$, and let $p_0,\dotsc,p_n$ be any prefixes of $w$ of strictly increasing length. Given a representation of the series of dimension $n$, $(\lambda, \mu, \rho)$ say, consider the vectors $\lambda \mu(p_0),\dotsc, \lambda \mu(p_n)$. Note that for some $i$, $\lambda \mu(p_i)$ is a linear combination of the vectors with smaller index. Let $s_i$ be the suffix of $w$ corresponding to $p_i$ and multiply the linear combination on the right by $\mu(s_i)$. Deduce that $\lambda \mu(w)$ is a linear combination of the vectors $\lambda \mu(p_j s_i)$ for $j<i$.

Finally, $\lambda \mu(w) \rho$ is nonzero (since $w\in\suppx{S}$), and hence so is one of the numbers $\lambda \mu(p_j s_i) \rho$. Conclude that the subword $p_j s_i$ of $w$ is also in the support of $S$.
\end{proof}

Let $S$ be a rational series over the doubled alphabet $\bxx$, with support included in the set of reduced words, and with well-ordered support (upon identifying the free group with the set of reduced words). We would like a method of determining $\hbox{\textsf{min}}\suppx{S}$. Our algorithm rests on the following key lemma.

\begin{lemm}\label{th: big bound}
There exists a bound $N$, depending on the rank $n$ of the series $S$ and on the cardinality of the alphabet, such that for each word $w$ in $\suppx{S}$ of length at least $N$, $w$ has a factor of the form $u = u_1\dotsb u_n$ with each $u_i  >1$ in the group ordering.
\end{lemm}

\begin{prop}\label{th: min-supp}
Given $S$ as above, $\mathsf{min}\suppx{S}$ is effectively computable.
\end{prop}

\begin{proof}
We claim that for each word $w$ in the support of $S$ of length at least $N$ (from Lemma \ref{th: big bound}), there exists a subword that is smaller for the group ordering and still in the support. This makes the problem finite: let $\mathcal N$ denote the set of words $w$ of length at most $N$, together with all of their subwords; restrict to the set $\mathcal N \cap \suppx{S}$, and compare these elements pairwise to find the minimum.
(See Appendix \ref{sec: Lyndon} for one way to make these comparisons.) 

To see the claim, suppose $w \doteq w_0u_1\dotsb u_nw_n$, with the $u_i>1$ as in Lemma \ref{th: big bound} and the $w_i$ possibly empty. Note that deleting any subset $I$ of the $u_i$ from $w$ results in a subword of $w$ that is smaller for the group ordering. (Indeed, $a<b \implies ua v<ub v$ for all $a,b,u,v\in \fg$.) Finally, put $p_j := w_0u_1\dotsb u_j$, and consider the subwords $p_js_i$ of $w$ from the proof of Proposition \ref{th: subwords}. One of these will belong to $\suppx{S}$.
\end{proof}

Thus determining whether $\suppx{S}=\emptyset$ is a finite problem, as we now indicate.

\begin{theo}[The Word Problem]\label{th: word problem} 
Given  $a\in \ff{X}$, the problem of determining whether $a=0$ has an effective solution.
\end{theo}

\begin{proof}
Given a rational expression for $a$, construct an equivalent $*$-rational expression $S$ over $S\in\red\XX$ according to \eqref{eq:a_inverse}. This may be done effectively, after Proposition \ref{th: min-supp}. Next, replace this $S$ by a rational series without simplification, as in Corollary \ref{th: MN effective}. 
Now the word problem is reduced to the well-known problem of determining if a rational series in $\monser{\bxx}$ is zero; this is solved using Sch\"utzenberger's reduction algorithm (see \cite[Sec. 2.3]{BerReu:2011}).

Keep at hand the representation $(\lambda,\mu,\rho)$ for $S$ developed there. Finally, pick each word $w$ of length at most $N$ (from Lemma \ref{th: big bound}), and determine if any of its subwords belong to $\suppx{S}$. (Recall $(S,u)$ is simply computed as $\lambda\mu(u)\rho$.) This completes the proof.
\end{proof}

\subsubsection{Proof of key lemma}
After some preparatory results of Jacob \cite{Jac:1980,Jac:1978}, we will be ready to prove  Lemma \ref{th: big bound}. 

\begin{prop}[Jacob, {\cite[Th. 3.5.1]{BerReu:2011}}]\label{th: Jacob N1}
Given a rational series $S\in \monser{A}$, there exists an integer $N_1$ such that for any word $w \in \suppx{S}$, and for any factorization $w = w_0uw_1$ satisfying $|u| \geq N_1$, there exists a factorization $u = pvs$ such that $\suppx{S} \cap \suppx{w_0pv^*sw_1}$ is infinite.
\end{prop}

Recall that an $n\times n$ matrix over $\K$ is called \demph{pseudo-regular} if it belongs to a subgroup of the multiplicative semigroup of matrices. Equivalently, if it is similar to a matrix of the form $\begin{pmatrix} g & 0 \\ 0  & 0\end{pmatrix}$ with $g\in \mathrm{GL}_{n'}(\K)$ and $0\leq n'\leq n$.

\begin{prop}[Jacob, {\cite[Exer. 3.5.1]{BerReu:2011}}]\label{th: Jacob N2}
Let $(\lambda,\mu,\rho)$ be a representation of a rational series $S\in\monser{A}$. For all integers $p$, there is a bound $N_2$, depending on $p$, the rank $n$ of $S$, and the cardinality of $A$, so that if $w$ is any word in $\suppx{S}$ of length at least $N_2$, then $w$ has consecutive factors $u_1,\dotsc,u_p$ with all $\mu(u_i)$ pseudo-regular (and having a common kernel and image).
\end{prop}

\begin{proof}[Proof of Lemma \ref{th: big bound}]
Let $S$ be a rational series whose support is included in $\red\XX$. We may assume that $\suppx{S}$ is well-ordered, upon identifying $\fgx{X}$ with $\red\XX$. Fix some representation $(\lambda, \mu, \rho)$ of $S$ of dimension $n$. 

Note that if $u,v',v''$ are elements of the free group with $u<1$, then $\suppx{v'u^*v''}$ has no smallest element (since this is true for $u^*$). The same is true even for each infinite subset of $\suppx{v'u^*v''}$.
From this we deduce that if $u,v',v''$ are words in $\red\XX$, then $\suppx{S} \cap \suppx{v'u^*v''}$ cannot be infinite. Now, it follows from the proof in \cite{BerReu:2011} of Proposition \ref{th: Jacob N1} that if $\mu(u)$ is pseudo-regular and ${v'uv''} \in \suppx{S}$, then $\suppx{v'u^*v''} \cap \suppx{S}$ is infinite. Thus, if $\mu(u)$ is pseudo-regular, and $v'uv'' \in\suppx{S}$, then $u>1$.

Using Proposition \ref{th: Jacob N2}, we see that if a word $w$ in $\suppx{S}$ is long enough, then it has a factor $u=u_1\dotsb u_n$ such that each $\mu(u_i)$ are pseudo-regular. Hence each $u_i>1$ and the lemma is proven.
\end{proof}

\begin{rema}
Jacob's bounds are extremely large. In \cite{Reu:1980}, see also \cite[Th. 1.12]{Okn:1998}, a common bound, better than $N_1$ and $N_2$, is given: $N(n) = \prod_{i=1}^n \left[\binom{n}{i}+1\right]$. However, it is still rather large (and proveably too large for $N_1$). It may be possible that one could further improve $N_2$ as well. 
\end{rema}


\appendix
\section{Using Euler's Identity}\label{sec: Euler}

In the introduction, we have claimed that Euler's identity is sometimes useful for verifying identities in the free skew field. Here we illustrate with \eqref{eq: rational identity} after first indicating how this identity was found.

\subsection{Skew-symmetry of quasi-Pl\"ucker coordinates}
Given an $n\times n$ matrix $A$, a row index $i$, and a column index $j$, the \demph{$(i,j)$-quasideterminant} is defined if the $(n{-}1)\times (n{-}1)$ submatrix $A^{i,j}$ is invertible. In that case, we put
\[
	\left|A\right|_{ij} = A_{i,j} - A_{i,[n]\setminus j} \cdot \left(A^{i,j}\right)^{-1} \!\cdot A_{[n]\setminus i,j}\,.
\]
Here, subscripts represent row and column indices of $A$ to keep (when building submatrices) and superscripts represent indices of $A$ to delete. For example,
\[
	\left|A\right|_{2,1} = \left|
	\begin{array}{@{}ccc@{}} a & b & c \\   \fbox{$u$} & v & w \\   x & y & z \end{array}
	\right| = u - \begin{pmatrix} v & w\end{pmatrix} \cdot \begin{pmatrix}b & c \\ y & z\end{pmatrix}^{-1} \!\!\cdot \begin{pmatrix}a \\ x \end{pmatrix}.
\]

Gelfand and Retakh \cite{GelRet:1991} introduced quasideterminants as a replacement for the determinant in noncommutative settings. 
Since that time, they have proven useful in a number of different settings \cite{GKLLRT:1995,MolRet:2004,DiK:2011}. Specific to the present discussion is the use of \emph{quasi-Pl\"ucker coordinates} to describe coordinate rings for quantum flags and Grassmannians \cite{Lau:2010}. 

Given an $n\times m$ matrix $A$ with $n<m$, fix a choice $K\subseteq[m]$ of $n-1$ columns of $A$, and let $i,j$ be two additional columns ($i\notin K$). Fix, also, a row $r$. The associated \demph{quasi-Pl\"ucker coordinate,} introduced by Gelfand and Retakh, is defined by 
\[
   p_{ij}^K = \left(\bigl|A_{[n],i\cup K}\bigr|_{r,i}\right)^{-1} \cdot \bigl|A_{[n],j\cup K}\bigr|_{r,i} \,,
\]
and is independent of $r$ (when the ratio is well-defined) \cite[Sec. 4]{GGRW:2005}. 
%
%
The quasi-Pl\"ucker coordinates reduce to ratios of classical Pl\"ucker coordinates 
in the commutative setting. As such, analogs of the celebrated Pl\"ucker relations, skew-symmetry relations, and more may be expected to hold. And they do. The skew-symmetry relations  \cite[Th. 4.4.1]{GGRW:2005} take the form 
\begin{equation}\label{eq: skew-symmetry}
	p_{ij}^{k\cup L}\cdot p_{jk}^{i\cup L} = - p_{ik}^{j\cup L} ,
\end{equation}
where $i,j,k,L$ indicate column indices of an $n\times m$ matrix, with $m>n$, $|L|=n-2$, and $\{i,j,k\}\cap L = \emptyset$. 
Identity \eqref{eq: rational identity} is a special instance of skew-symmetry, using the $2\times3$ matrix
\[
		A= \begin{pmatrix}
			1 & x & 1 \\
			y & 1 & z 
		\end{pmatrix}.
\]
(Here $L=\emptyset$.) Taking $(i,j,k)=(1,2,3)$, the skew-symmetry relation \eqref{eq: skew-symmetry} becomes
\begin{gather*}
	\left|\begin{array}{cc}\fbox{$1$}&1\\y&z\end{array}\right|^{-1}
		\left|\begin{array}{cc}\fbox{$x\rule[0ex]{0pt}{1.4ex}$}&1\\1&z\end{array}\right|
\,\cdot\,	\left|\begin{array}{cc}x&1\\\fbox{$1$}&y\end{array}\right|^{-1}
		\left|\begin{array}{cc}1&1\\\fbox{$z\rule[0ex]{0pt}{1.4ex}$}&y\end{array}\right|
\ =\ -	\left|\begin{array}{cc}1&x\\\fbox{$y$}&1\end{array}\right|^{-1}
		\left|\begin{array}{cc}1&x\\\fbox{$z\rule[-.2ex]{0pt}{1.4ex}$}&1\end{array}\right| ,
\intertext{or}
\big(y\inv-z\inv\big)\inv\big(x-z\inv\big)\big(1-yx\big)\inv\big(y-z\big)  
\ =\ - \big(1-x\inv y\inv \big)\inv\big(x\inv-z\big) .
\end{gather*}

\subsection{Using Euler}\label{sec: using euler}

Using geometric series, the star notation, and our barred variables shorthand, we may rewrite \eqref{eq: rational identity} (and the above) as
\[
(x-\overline z) \,{(yx)}^* \, (y-z)  
\ = \ - (\overline y-\overline z)\, {(\overline x\,\overline y)}^* \, (\overline x-z) .
\]
We leave it to the reader to verify this identity by distributing products and equating terms. (\emph{Hint:} implicit in Euler's identity is the equality ${\overline a}^* = -aa^*$.)


\section{Extensions \& Open Problems}\label{sec: Extensions}
We collect some possible extensions of the results presented in Section \ref{sec: rank}, as well as some open problems.

\subsection{Extending the main theorem} In what generality does Theorem \ref{th: new_connes_criterion} hold?

Let us replace the field $\K$ by any topological ring $R$ and say that the Cauchy product $a\cdot b$ is well-defined for $a,b\in R^\fg$ if
\begin{gather*}
(ab,\w) := \sum_{\substack{\alpha,\beta\in \fg \\ \alpha\beta=\w}} (a,\alpha)(b,\beta)
\end{gather*}
converges (in $R$) for all $\w\in \fg$. 
\begin{prob} 
Do Proposition \ref{th: connes identities} and Lemma \ref{th: keylemma}, which use the discrete topology on $\K$, hold in this more general setting? 
\end{prob}

Let $R$ be a (not necessarily commutative) ring. A free $R$-module $M$ has \emph{invariant basis number} (IBN) if any two bases of $M$ have the same cardinality. If every free $R$-module has this property, we say that $R$ is an \demph{IBN ring.}\footnote{Commutative rings, noetherian rings and division rings are among the chief examples. See \cite[Ch. 1.4]{Coh:1995} for more information.} 

\begin{prob} 
Does Lemma \ref{th: closed} hold if the field $\K$ is replaced by an IBN ring $R$?
\end{prob}

It would seem, then, that the proofs of Theorem \ref{th: new_connes_criterion} and Proposition \ref{th:rank bounds} extend to the setting of topological IBN rings. 

\begin{prob}\label{th: extension} Let $R$ be any topological IBN ring. Is it true that an element $a\in R^\fg$ belongs to the rational closure of $R\fg$ in $R^\fg$ iff the operator $\FF{a}$ is of finite rank?
\end{prob}

\subsection{Rank and image of the Connes operators}

Hand calculations suggest that the bounds given in Proposition \ref{th:rank bounds} are far from tight beyond rank one.

\begin{prob} Find a family of examples showing the bound is tight beyond rank one, or find a different tight bound beyond rank one. 
\end{prob}

Using \eqref{eq: Fab} and \eqref{eq: Fa*}, hand calculations further suggest that the image of $\FF{a}$ is spanned by 
\[
\Bigl\{\FF{a}\{\pp{\w},\w\} : \overline{\w}\in\suff{a}\Bigr\} ,
\]
where $\suff{a}$ denotes the (finite) set of all suffixes of all monomials used in a $\ast$-rational expression for $a$. 

\begin{prob}
Determine if this is indeed the case. 
\end{prob}

\begin{rema}
A similar problem could be posed for the Hankel operator $\circc a$, with the answer perhaps being found ``between the lines'' of existing proofs of the equivalence of \emph{rational} and \emph{representable} series.
\end{rema}


\section{Lyndon Words in the Magnus Ordering}\label{sec: Lyndon}

The results in this appendix are not needed for any other result in the paper. Still, they were found during the search for a proof of Proposition \ref{th: min-supp}, so this paper seems like the best place to share them. Our main result is Theorem \ref{th: compare lyndons}; the corresponding algorithm appears in Section \ref{sec: lyndon representations}.

\subsection{The Magnus ordering}
The standard proof that a free group may be ordered computes its lower central series, orders the corresponding (free abelian) quotient groups, then patches these orderings together to build an ordering of the free group \cite[Th. 7.3]{Sak:2009}. However, there is a very elementary means of ordering $\fgx{X}$ that requires only that the mapping $\Mag: \fgx{X} \to \Z\llangle X\rrangle$ given by
\[
	x \mapsto 1+x \ \ (\hbox{for all }x\in X)
\]
is an embedding. This is proven in \cite[Th. 5.6]{MagKarSol:1966}. 

Let $<$ be a total ordering on the free monoid $X^*$. (We use here the \demph{military ordering:} length plus lexicographic.) Extend $<$ to a total ordering of $\Z\llangle X\rrangle$ by using the fact that $\Z$ is well-ordered. Putting $x<y$, we have, e.g., 
\[
	-3x -2y + yx < -3x + yx < y -2yx < y + 2xy
\]
(the first inequality because $-2y < 0y$, the second because $-3x < 0x$, the third because $xy<yx$ and $0xy < 2xy$). Finally, to order $\w,\uu \in \fg$, say $\w\prec \uu \iff \Mag(\w) < \Mag(\uu)$. Thus, e.g., $x \succ y $ because
\[
	M(x) = 1 + 1x + 0y \qand M(y) = 1 +0x + 1y.
\]
Similarly, the reader may verify that $xy^{-1} \prec y^{-1}x$.

\begin{rema}
While the ordering of $\fgx{X}$ just described is customarily called the \emph{Magnus ordering,} the first proof we have been able to find belongs to Bergman \cite{Berg:1990}. 
\end{rema}

\subsection{The subword function}
Given an alphabet $A$, and words $w,v \in A^*$, the \demph{subword function} $\binom{w}{v}$ returns the number of occurrences of $v$ as a subword (i.e., subsequence) of $w$ \cite{Eil:1974}. Note that $\binom{x^n}{x^l} = \binom{n}{l}$, so this is a natural generalization of binomial coefficients to words. 

We make a connection to the Magnus transformation to extend the first argument of the subword function to the free group. If $\w\in X^*$, it is clear that 
\[
	\Mag(\w) = \sum_{v\subseteq \w} \binom{\w}{v} v \,.
\]
See \cite[Prop. 6.3.6]{Lot:1997}. For $\w\in \fgx{X}$ and $v\in X^*$, define $\binom{\w}{v}$ to be the coefficient of $v$ in $\Mag(\w)$. The extension of $\binom{\raisebox{-.25\height}{$\,\arg\,$}}{v}$ from $X^*$ to $\fg$ is continuous with respect to the \emph{profinite topology} on $\fg$.%
\footnote{This is the coarsest topology such that all homomorphisms $\phi \colon \fg \to G$ to finite, discrete groups are continuous. See \cite{Hall:1950} or \cite{MelReu:1993} for details.} 
The following evident result will be useful in what follows.

\begin{lemm}\label{th: subword-check} Let $1=v_0<v_1<\dotsb$ be the total ordering of $X^*$. If $\w,\w'\in\fgx{X}$, then $\w \prec \w'$ if and only if there is some $p\in\N$ with
\[
	\binom{\w}{v_i} = \binom{\w'}{v_i} \ \hbox{for all }i<p, \hbox{ and }\ 
	\binom{\w}{v_p} < \binom{\w'}{v_p} .
\]
\end{lemm}

\subsection{Lyndon words}
Recall that a word $\ell \in X^*$ is \demph{Lyndon} if it is lexicographically smaller than all of its cyclic permutations.  Every word $v\in (X^*,<)$ has a unique decomposition of the form
\[
	v = \ell_1^{m_1} \ell_2^{m_2} \dotsb \ell_r^{m_r} \qquad (\ell_1 > \ell_2 > \dotsb > \ell_r, \ m_j\geq1) ,
\]
where the $\ell_i$ are Lyndon \cite[Ch. 5]{Lot:1997}. 

Our main result is a restriction of the search space in Lemma \ref{th: subword-check} from all words in $X^*$ to Lyndon words.

\begin{theo}\label{th: compare lyndons}
Let $\ell_1< \ell_2<\dotsb$ denote the Lyndon words in $X^*$, ordered with respect to the military ordering. If $\w,\w'\in\fgx{X}$, then with respect to the Magnus ordering, $\w \prec \w'$ if and only if there is some $p\in\N$ with
\[
	\binom{\w}{\ell_i} = \binom{\w'}{\ell_i} \ \hbox{for all }i<p, \hbox{ and }\ 
	\binom{\w}{\ell_p} < \binom{\w'}{\ell_p} .
\]
\end{theo}

To prove this theorem, we show that the subword function $\binom{\w}{v}$ on $\fgx{X}$ may be computed using only its values on Lyndon words $\binom{\w}{\ell}$. 
As an illustration, given any $\w \in X^*$ and $x<y\in X$, a simple computation shows that
\begin{gather*}
	\binom{\w}{y}\!\binom{\w}{x}  = \binom{\w}{yx} + \binom{\w}{xy} \qand
	\binom{\w}{xy}\!\binom{\w}{x} = \binom{\w}{xyx} + 2\binom{\w}{xxy} + \binom{\w}{xy} ,
\end{gather*}
so knowing the values of $\binom{\w}{\raisebox{.25\height}{$\,\arg\,$}}$ on the Lyndon words $x, y, xy$, and $xxy$ is enough to recover the values $\binom{\w}{yx}$ and $\binom{\w}{xyx}$. 
%
\begin{rema}
After formulating our proof of this statement, we discovered the same proof (for $X^*$) in the unpublished thesis of P\'eladeau \cite[Prop. 5.2.17]{Pel:1986}.
\end{rema}
%
To see the above examples through to a proof of Theorem \ref{th: compare lyndons}, we need 
the \emph{infiltration product} of Chen--Fox--Lyndon. (Briefly, $u \infl v$ is ``shuffle, plus overlap''. See \cite[Ch. 6]{Lot:1997} for details.) For example, 
\[
	xy \infl x = 2xxy + xyx + xy .
\]
In general, the leading term in $u \infl v$ is the same as that in $u \shuf v$ for our ordering of $X^*$. Now, \cite[Th. 3.9]{CFL:1958}\footnote{The recapitulation in \cite[Ch. 6]{Lot:1997} takes $\w \in X^*$, but the original theorem is stated in the generality that we need here.} has that for all $\w \in \fgx{X}$ and $t,u\in X^*$,
\begin{equation}\label{eq: CFL}
	\binom{\w}{t}\!\binom{\w}{u} = \sum_{v\in X^*} (t\infl u,v)\binom{\w}{v} .
\end{equation}

We also need the following result of Radford \cite[Th. 3.1.1]{Rad:1979} relating the concatenation product and shuffle product $\shuf$.

\begin{lemm}\label{th: Radford} 
If $\ell_1^{m_1} \dotsb \ell_r^{m_r}$ is the Lyndon decomposition of a word $v\in X^*$, then 
\[
	\frac{1}{m_1!\dotsb m_r!} \, \ell_1^{\shuf m_1} \shuf \dotsb \shuf  \ell_r^{\shuf m_r} \ = \  v \ + \ (\hbox{smaller words w.r.t. }{<}). 
\]
\end{lemm}

\begin{proof}[Proof of Theorem \ref{th: compare lyndons}]
After Lemma \ref{th: subword-check}, we may compare $\w$ and $\w'$ by comparing the functions $\binom{\w}{\raisebox{.25\height}{$\,\arg\,$}}$ and $\binom{\w'}{\raisebox{.25\height}{$\,\arg\,$}}$. Given $v\in X^*$, compute its Lyndon decomposition $\ell_{i_1}\dotsb \ell_{i_r}$ (with $\ell_{i_j}\geq \ell_{i_{k}}$ for all $j<k$). Using \eqref{eq: CFL} and Lemma \ref{th: Radford}, we have
\[
	\binom{\raisebox{-.25\height}{$\,\arg\,$}}{v} = 
	\prod_{j}\binom{\raisebox{-.25\height}{$\,\arg\,$}}{\ell_{i_j}} 
	- \bigl(\hbox{\it terms $\displaystyle \binom{\raisebox{-.25\height}{$\,\arg\,$}}{u}$ with $u<v$}\bigr).
\]
Now, either $\Mag(\w)$ and $\Mag(\w')$ agree on these words $u$, in which case they have no bearing on the comparison of $\binom{\w}{v}$ and $\binom{\w'}{v}$, or they differ, in which case the determination of whether or not $\w \prec \w'$ has already been made. In any case, induction completes the proof. 
\end{proof}

\subsection{Effective computation}\label{sec: lyndon representations}
Finally, we indicate via example how to effectively compute $\binom{\w}{\ell}$ for $\w \in \fgx{X}$ and $\ell$ a Lyndon word in $X^*$. Consider $\ell = xy$. Note that
\[
	xy \shuf X^* = \sum_{w\in X^*} \binom{w}{xy} w \,.
\]
An automaton recognizing the series $xy \shuf X^*$ is shown in Figure \ref{fig: subword computation}. 
\begin{figure}[!hbt]
  \centering
    \unitlength=3.25pt
    \begin{picture}(30, 15)(0,0)
    \gasset{Nw=5,Nh=5,Nmr=2.5,curvedepth=0,loopdiam=6}
    \thinlines
    \node[Nmarks=i,iangle=180](A1)(0,10){{\small\sf0}}
    \drawloop[loopangle=270](A1){$x,y$}
    \node(A2)(15,10){{\small\sf1}}
    \drawloop[loopangle=270](A2){$x,y$}
    \node(A3)(30,10){{\small\sf2}}
    \rmark[rdist=.4](A3)
    \drawloop[loopangle=270](A3){$x,y$}
    \drawedge(A1,A2){$x$}
    \drawedge(A2,A3){$y$}
    \end{picture}
\caption{An automaton recognizing the series $xy\shuf X^*$.}
\label{fig: subword computation}
\end{figure}
An equivalent representation $(\lambda,\mu,\rho)$ of the series is as follows:
\[
	\lambda = \begin{pmatrix} 1 & 0 & 0 \end{pmatrix} \!,	
	\quad
	\mu(x) = \begin{pmatrix} 1 & 1 & 0\\ 0 & 1 & 0\\ 0 & 0 & 1\end{pmatrix} \!,
	\quad
	\mu(y) = \begin{pmatrix} 1 & 0 & 0\\ 0 & 1 & 1\\ 0 & 0 & 1\end{pmatrix} \!,
	\qand
	\rho = \begin{pmatrix}0\\ 0\\ 1\end{pmatrix} \!.
\]
Thus $\binom{w}{xy}$ is recovered by taking the $(1,3)$ entry of $\mu(w)$. By continuity, the same is true of $\binom{\w}{xy}$ for any $\w\in\fg$. Since there are effective algorithms for producing automata and representations, this gives an effective means of determining whether or not $\w \prec \w'$ in $\fg$.

\newcommand{\etalchar}[1]{$^{#1}$}
\def\cprime{$'$}
\providecommand{\bysame}{\leavevmode\hbox to3em{\hrulefill}\thinspace}
\providecommand{\MR}{\relax\ifhmode\unskip\space\fi MR }
\providecommand{\MRhref}[2]{%
  \href{http://www.ams.org/mathscinet-getitem?mr=#1}{#2}
}
\providecommand{\href}[2]{#2}

\end{document}